\documentclass[12pt, a4paper]{amsart}
\usepackage{amsmath}
\usepackage{amsthm}
\usepackage{amsfonts}
\usepackage{amssymb}
\usepackage{latexsym}
\usepackage{enumerate}
\usepackage{graphicx}

\newcommand{\Z}{\mbox{$\mathbb Z$}}

\newcommand{\N}{\mbox{$\mathbb N$}}     

\newtheorem{theo}{Theorem}[section]

\newtheorem{lem}[theo]{Lemma}
\newtheorem{prop}[theo]{Proposition}

\begin{document}

\title{The sum of digits of $n$ and $n^2$}

\author{Kevin G. Hare}
\address{Department of Pure Mathematics,  University of Waterloo, Waterloo, Ontario,  Canada,  N2L 3G1,}
\email{kghare@math.uwaterloo.ca}
\thanks{K.G. Hare was partially supported by NSERC}
\thanks{Computational support provided by CFI/OIT grant}

\author{Shanta Laishram}
\address{Department of  Mathematics,
Indian Institute of Science Education and Research, Bhopal, 462 023, India,}
\email{shanta@iiserbhopal.ac.in}

\author{Thomas Stoll}
\address{Institut de Math\'ematiques de Luminy, Universit\'e de la M\'editerran\'ee, 13288 Marseille Cedex 9, France,}
\email{stoll@iml.univ-mrs.fr}
\thanks{Th. Stoll was partially supported by an APART grant of the Austrian Academy of Sciences.}


\maketitle

\begin{abstract}
  Let $s_q(n)$ denote the sum of the digits in the $q$-ary expansion of an integer $n$.
  In 2005, Melfi examined the structure of $n$ such that $s_2(n) = s_2(n^2)$.
  We extend this study to the more general case of generic $q$ and polynomials $p(n)$,
  and obtain, in particular, a refinement of Melfi's result. We also give a more detailed analysis of the special case $p(n) = n^2$, looking at
  the subsets of $n$ where $s_q(n) = s_q(n^2) = k$ for fixed $k$.
\end{abstract}

\section{Introduction}

Let $q\geq 2$ and denote by $s_q(n)$ the sum of digits in the
$q$-ary representation of an integer $n$. Recently, considerable
progress has been made towards understanding the interplay
between the sum-of-digits of some algebraically defined sequences,
such as primes~\cite{MR09-1} and polynomials~\cite{DT06} or, in
particular, squares~\cite{MR09-2}. In the latter, C.~Mauduit and
J.~Rivat proved an asymptotic expansion of the sum of digits of
squares~\cite{MR09-2} in arithmetic progressions. Their proof
heavily relies on good estimates of quadratic Gauss sums. For the
case of general polynomials $p(n)$ of degree $h> 2$ there is still
a great lack of knowledge regarding their distribution with
respect to digitally defined functionals~\cite{DT06}.

Several authors studied the pointwise properties and relationships
of $s_q(p(n))$, e.g., K.~Stolarsky~\cite{St78},
B.~Lindstr\"om~\cite{Li97}, G.~Melfi~\cite{Me05}, and M.~Drmota
and J.~Rivat~\cite{DR05}. In particular, a conjecture of
Stolarsky~\cite{St78} about some extremal distribution properties
of the ratio $s_q(p(n))/s_q(n)$ has been recently settled by the
authors~\cite{HLS10}. Melfi~\cite{Me05} proposed to study the set
of $n$'s such that $s_2(n^2)=s_2(n)$, and he obtained that
\begin{equation}\label{melfi}
  \#\left\{n<N : \quad s_{2}(n^2)=s_2(n)\right\}\gg N^{1/40}.
\end{equation}
Using heuristic arguments, Melfi conjectured a much stronger
result that
\begin{equation}
\#\left\{n<N : \quad s_{2}(n^2)=s_2(n)\right\}
    \approx \frac{N^\beta}{\log N}
\end{equation}
with $\beta \approx 0.75488\dots$, giving an explicit formula for
$\beta$. The aim of the present paper is to provide a
generalization to general $p(n)$ and base $q$ of Melfi's result as
well as to use the method of proof to sharpen Melfi's exponent in~(\ref{melfi}).
Moreover, we provide a local analogon, i.e., getting a lower bound
for the number of $n$'s such that $s_q(n^2)=s_q(n)=k$ for some
fixed $k$.

\begin{theo}\label{mtheo3}
Let $p(x) \in \Z[x]$ have degree at least 2, and positive leading
coefficient. Then there exists an explicitly computable $\gamma>0$,
dependent only on $q$ and $p(x)$, such that
\begin{equation}\label{thm3estim}
  \#\left\{n<N,\;q\nmid n: \quad \vert s_q(p(n))-s_q(n)\vert \leq
\frac{q-1}{2}\;\right\}\gg N^\gamma,
\end{equation} where the implied constant depends only on $q$ and
$p(x)$.
\end{theo}

This result is given in Section \ref{sec:mtheo3}. In the general case of
$q$-ary digits and polynomials $p(x)$, the
bound $(q-1)/2$ in~(\ref{thm3estim}) cannot be improved. This is
easily seen by recalling the well-known fact
\begin{equation}\label{wellknown}
  s_q(n)\equiv n \bmod (q-1).
\end{equation}
Indeed, if we set $p(x)=(q-1)x^2+x+a$ for $a\in\N$ then we find
that
$$ s_q(p(n))-s_q(n)  \equiv p(n)-n \equiv a \bmod
(q-1)$$
which could be any of $0,1,\ldots,q-2$ depending only on
the choice of $a$.

The method of proof of Theorem~(\ref{mtheo3}) allows to improve
on Melfi's result~(\ref{melfi}).

\begin{theo}\label{mtheo3+}
\begin{equation}\label{melfi+}
  \#\left\{n<N : \quad s_{2}(n^2)=s_2(n)\right\}\gg N^{1/19}.
\end{equation}
\end{theo}

Following on Melfi's paper~\cite{Me05}, we examine the case when $p(n) =
n^2$ and $q=2$ in more detail. We consider the set of all $n$'s
such that $s_2(n) = s_2(n^2)$, and partition the set into the
subsets dependent upon the value of $s_2(n)$. By noticing that
$s_2(n) = s_2(2 n)$ and $s_2(n^2) = s_2((2n)^2)$ we see that we
can restrict our attention to odd $n$.
\begin{theo}\label{mtheo4}
  Let $k\leq 8$. Then
   $$\{n<N,\; n \mbox{ odd}: \quad s_2(n^2)=s_2(n)=k\}$$
   is a finite set.
\end{theo}
This was done by explicit computation of all such $n$ which are
given in Tables \ref{tab:k=1..7} and \ref{tab:k=8}. A discussion
of how these computations were made is given in Section
\ref{sec:mtheo4}.

Based on these initial small values of $k$, one might expect that this is always
true. Let
\begin{equation}\label{casek12}
  n_{(2)} = 1101111 \underbrace{00\dots 00}_{r} 1101111
\end{equation}
be written in base $2$. Then $s_2(n) = s_2(n^2) = 12$ for all $r \geq 8$.
This is in fact a special case of a more general property.
\begin{theo}\label{mtheo5}
  Let $k\geq 16$ or $k \in \{12, 13\}.$ Then
   $$\{n<N,\; n \mbox{ odd}: \quad s_2(n^2)=s_2(n)=k\}$$
   is an infinite set.
\end{theo}
The proof of this result is given in Section \ref{sec:mtheo5}.
Despite of great effort we are not able to decide the finiteness
problem in the remaining cases $k\in\{9,10,11,14,15\}$. However,
we will comment on some heuristic evidence that it seems unlikely
that there are infinitely many solutions in the cases $k=9$ and
$k=10$, respectively, in Section \ref{sec:comp2}.

Somewhat surprisingly, a similar answer can be given if $q
\geq 3$.

\begin{theo}\label{genqsquares}
  Let $q\geq 3$ and assume
  $$k\geq 94(q-1).$$
  Then the equation
  \begin{equation}\label{Dioeq}
    s_q(n^2) = s_q(n) = k
  \end{equation}
  has infinitely many solutions in $n$ with $q\nmid n$ if and only if
  \begin{equation}\label{knec}
    k (k-1) \equiv 0\quad   \bmod (q-1).
  \end{equation}
\end{theo}

We show this result in Section~\ref{sec:genqsquares}.

\section{Proof of Theorems~\ref{mtheo3} and~\ref{mtheo3+}}\label{sec:mtheo3}

Following Lindstr\"om~\cite{Li97} we say that terms are
\textit{noninterfering} if we can use the following splitting
formul\ae:
\begin{prop}\label{propsplit}
For $1\leq b<q^k$ and $a,k\geq 1$,
  \begin{align}
    s_q(aq^k+b)&=s_q(a)+s_q(b),\label{splitpos}\\
    s_q(aq^k-b)&=s_q(a-1)+(q-1)k-s_q(b-1).\label{splitneg}
  \end{align}
\end{prop}
\begin{proof} See~\cite{HLS10}.
\end{proof}

\textit{Proof of Theorem~\ref{mtheo3}:} The proof uses a construction of a
sequence with noninterfering terms which has already been used
in~\cite{HLS10}. However, to obtain the bound $N^\gamma$ in~(\ref{thm3estim}) instead of
a logarithmic bound, we have to make a delicate refinement. To begin with, define the polynomial
$$t_m(x)=m x^4+mx^3-x^2+mx+m$$
where $m\in \Z$. Set $m=q^l-r$ with $1\leq r\leq \lfloor q^{\alpha l}\rfloor$, $q \nmid r$ and
$0<\alpha<1$. Obviously, for $\alpha<1$ there exists $l_0(\alpha)$ such that for all $l>l_0(\alpha)$ we have $m\geq 3$.
Furthermore let $k$ be such that $q^k>m$. By consecutively
employing~(\ref{splitpos}) and~(\ref{splitneg}) we see that
\begin{align}
  s_q(t_m(q^k))&=(q-1)k+s_q(m-1)+3s_q (m)\nonumber\\
  &=(q-1)k+s_q (q^l-(r+1))+3 s_q(q^l-r)\nonumber\\
  &=(q-1)k+(q-1)l -s_q(r) +3((q-1)l -s_q(r-1))\label{crx}\\
  &\leq (q-1)k+4(q-1)l.\nonumber
\end{align}
First consider the easier case of monomials $p(n)=n^h$, $h\geq 2$  where we can give a somewhat more direct proof. We have
\begin{align}
  t_m(x)^h&=(mx^4+mx^3-x^2+mx+m)^h\nonumber\\
  &=\sum_{j=0}^{4h} c_{j,h}(m) x^j\label{cdefin}\\
  &=m^h x^{4h}+hm^h x^{4h-1}+\left(\binom{h}{2} m^h-h m^{h-1}\right) x^{4h-2}\nonumber\\
  &\qquad +\left(\left(h+\binom{h}{3}\right)m^h-2\binom{h}{2}m^{h-1}\right) x^{4h-3}+\mbox{smaller powers.}\nonumber
\end{align}
From~\cite{HLS10} we have that $t_m(x)^h$ has only positive
coefficients, which are bounded by $(2mh)^h$. This means that $s_q(t_m(q^k)^h)$ does not depend on $k$
if $k$ is sufficiently large (see~(\ref{splitpos})). More precisely, if $q^k>(2mh)^h$ (note that a sufficient condition for this is $k\geq (h+1)l$)
then we get for sufficiently large $l$
and a symmetry argument for the coefficients of $t_m(x)^h$,
\begin{align}
  s_q(t_m(q^k)^h)&\geq 2\left(s_q(m^h) +s_q(hm^h) +s_q\left(\binom{h}{2} m^h-h m^{h-1}\right)\right.\nonumber\\
  &\left.\qquad\qquad +s_q\left(\left(h+\binom{h}{3}\right)m^h-2\binom{h}{2}m^{h-1}\right)\right).\label{lowbound}
\end{align}
Consider the first summand $s_q(m^h)$ in~(\ref{lowbound}). We have
\begin{eqnarray}
  m^h & = & (q^l-r)^h = \sum_{j=0}^h \binom{h}{j} (-1)^{h-j} q^{jl} r^{h-j}
   \nonumber \\
      & = & \sum_{j=0}^h (-1)^{h-j} d_j q^{jl} \label{mhexpand}
\end{eqnarray}
which shows that $m^h$ is a polynomial in $q^l$ with coefficients of alternating signs. Now there are exactly $\lfloor h/2\rfloor$
negative signs in this expansion. All coefficients in~(\ref{mhexpand}) are bounded in modulus by
$$0 < d_j \leq (2r)^h\leq (2q^{\alpha l})^h\leq q^{(\alpha l+1)h},$$
and in turn their $q$-ary sum of digits is less than $s_q(d_j) \leq (q-1)(\alpha l+1)h$. Therefore, by~(\ref{splitneg}), we get that for
fixed $\alpha<1/h$ and sufficiently large $l$ we have
\begin{align}\label{hestim}
  s_q(m^h) &\geq \lceil h/2 \rceil (q-1)l-\lceil h/2 \rceil (q-1)(\alpha l+1)h\nonumber\\
  & \geq \frac{h}{2}(q-1)(l(1-\alpha h)-h).
\end{align}
A similar argument can be applied to the other three summands in~(\ref{lowbound}). This yields
\begin{equation}\label{lowbound2}
  s_q(t_m(q^k)^h)\geq 4h(q-1)(l(1-\alpha h)-h).
\end{equation}
Therefore, for each sufficiently large $l$ we can find $k$ with
\begin{equation}\label{thm3estimk}
  \vert s_q(t_m(q^k)^h)-s_q(t_m(q^k))\vert \leq \frac{q-1}{2}
\end{equation}
provided that $k \geq (h+1)l$ and
$$
  (q-1)k+4(q-1)l \leq 4h(q-1)(l(1-\alpha h)-h).
$$
Note that these conditions allow to successively increase $k$ (see~(\ref{crx})) in order that $s_q(t_m(q^k))$ and
$s_q(t_m(q^k)^h)$ differ by at most $(q-1)/2$. For sufficiently large $l$ these two conditions translate into
\begin{equation}\label{alphacond}
  (h+1)l\leq 4l(h(1-\alpha h)-1)-4h^2.
\end{equation}
Take $\alpha=1/(5h^2)<1/h$. It is then a direct calculation to verify that~(\ref{alphacond}) is true for all $h$ and sufficiently large $l$.
 Summing up, we have obtained that for sufficiently large $l$ we can find
$\gg q^{\alpha l}$ values $r$ where we in turn can provide a value $k$ with~(\ref{thm3estimk}). In addition, each triple
$(l,r,k)$ gives rise to a different value of $t_m(q^k)$. We thus have~(\ref{thm3estim}).

Now consider the case of a general polynomial $p(x)=a_h
x^h+a_{h-1} x^{h-1} + \dots + a_0 \in \Z[x]$. There exist positive
integers $s_1$ and $s_2$, both only depending on the polynomial
$p(x)$ such that $$p(q^{s_1}x+q^{s_2}+1)=a'_h x^h+a'_{h-1} x^{h-1}
+ \dots + a'_0$$ has only positive coefficients. With the notation of~(\ref{cdefin}) we obtain
\begin{align}\label{8terms}
p(q^{s_1} t_m(x)+q^{s_2}+1)&=\sum_{i=0}^3 a'_h c_{4h-i,h}(m) \; x^{4h-i}\\
&+\sum_{i=4}^7 \left(a'_h c_{4h-i,h}(m)+a'_{h-1} c_{4h-i,h-1}(m)\right)
x^{4h-i}\nonumber\\
&+\mbox{smaller powers.}\nonumber
\end{align}
First suppose $h\geq 4$. By choosing $s_1$ sufficiently large
(this choice again only depends on $p(x)$) we get that the
coefficients of $x^j$ in $p(q^{s_1} t_m(x)+q^{s_2}+1)$ with
$4h-7\leq j\leq 4h$ are polynomials in $m$ of degree $h$ since we can avoid unwanted cancellation
for these coefficients. The coefficients of these terms (as polynomials in $m$) are alternating in sign, since for
$h\geq 4$ and $i=0,1,\dots, 2h-1$ we have
\begin{align}\label{i4h-i}
c_{i,h}(m)=c_{4h-i,h}(m)=\sum_{j=h-\lfloor i/2\rfloor}^h
d_{j,i,h}m^j
\end{align}
where $d_{j,i,h} d_{j+1,i,h}<0$ for all $j$ with $h-\lfloor
i/2\rfloor\leq j< h$. Setting $m=q^l-r$ we therefore can choose $s_1$, $s_2$ in the
way that $a'_h c_{4h-i,h}(m)+a'_{h-1} c_{4h-i,h-1}(m)$ as a polynomial in $q^l$ has $\lceil
h/2 \rceil$ negative coefficients for each $=0,1,\ldots,2h-1$. Now, for $q^{s_2}+1<q^{s_1}$,
we get by~(\ref{crx}) that
$$s_q(q^{s_1}t_m(q^k)+q^{s_2}+1)\leq (q-1)k +4(q-1)l+2.$$
In~(\ref{8terms}) we have therefore found eight summands sharing the property of the eight summands in the monomial case
(see~(\ref{lowbound})). From this we proceed as as in the case of monomials to get the statement.

It remains to deal with the cases of general quadratic and cubic
polynomials, where we cannot directly resort to~(\ref{i4h-i}) (note that $8>(2h-1)+1$ for $h=2, 3$).
We instead do a more direct calculation. Let $h=\deg
p=2$ which is the case of quadratic polynomials. By suitably
shifting the argument $x\mapsto q^{s_1}x+q^{s_2}+1$ we can arrange
for a polynomial $p(q^{s_1}x+q^{s_2}+1)=a'_2 x^2+a'_1x+a'_0$ with
$a'_2, a'_1, a'_0>0$ and $2a'_2>a'_1$. Each coefficient of $x^i$
in $p(q^{s_1}t_m(x)+q^{s_2}+1)$, $0\leq i\leq 8$, is a function of
$m$ and of $a'_2, a'_1$ and $a'_0$. In a similar way as before (here we use $9$ summands
instead of the $8$ in the case of $h\geq 4$)
we obtain for sufficiently large $l$,
$$s_q(p(q^{s_1}t_m(q^k)+q^{s_2}+1))> 8 (q-1)l\geq 4h(q-1)l.$$
Now we can choose $k$ suitably to get the assertion. Finally,
for a cubic polynomial, we are able to achieve
$p(q^{s_1}x+q^{s_2}+1)=a'_3x^3+a'_2 x^2+a'_1x+a'_0$
with $a'_3, a'_2, a'_1, a'_0>0$ and $3a'_3>a'_2$. Then, each
coefficient of $x^i$ in $p(q^{s_1}t_m(x)+q^{s_2}+1)$, $0\leq i\leq 12$,
is a function of $m$ and $a'_3, a'_2, a'_1, a'_0$, and thus
we get for sufficiently large $l$,
$$s_q(p(q^{s_1}t_m(q^k)+q^{s_2}+1))> 12 (q-1)l\geq 4h (q-1)l.$$
By choosing $k$ suitably, we obtain the result.
This completes the proof of Theorem~\ref{mtheo3}.\hfill\qed

\medskip

\textit{Proof of Theorem~\ref{mtheo3+}:} We apply the method of proof of
Theorem~\ref{mtheo3} to the special case $q = 2$ and $p(n) = n^2$. Instead of
using the rather crude bounds, we here use exact values to get our result. To begin with,
we observe that the largest coefficient (as $m\to \infty$) of $t_m(x)^2$ is the coefficient of $x^4$, namely $4m^2+1$.
Therefore we get noninterfering terms when $2^k \geq 4m^2+1$. A sufficient
condition for this is $2^k\geq 4 \cdot 2^{2l}=2^{2l+2}$, or equivalently,
\begin{equation}\label{2kllow}
  k\geq 2l+2.
\end{equation}
On the other hand, the coefficients of $x^8$ and $x^7$ (resp. $x^1$ and $x^0$) in $t_m(x)^2$ are
$m^2$ and $2m^2$ which have the same binary sum of digits. Now assume
$\alpha<1/2$ and $l>l_0(\alpha)$ be sufficiently large.
We then use Proposition~\ref{propsplit} and set $m=2^l-r$ with $1\leq r\leq \lfloor 2^{\alpha l}\rfloor$ to obtain
\begin{align}\label{2expans}
  s_2(t_m(2^k)^2) &\geq 4 s_2(m^2)+s_2(4m^2+1)\\
  &=5 s_2\left((2^{l-1}-r )2^{l+1}+r^2\right)+1\nonumber\\
  &\geq 5 s_2(2^{l-1}-r)\nonumber\\
  &= 5\left((l-1) - s_2(r-1)\right)\nonumber\\
  &\geq 5(l-1) -5 \alpha l\nonumber\\
  &\geq (2+\varepsilon) l\nonumber
\end{align}
for any $0<\varepsilon<1/2$. This means that for any $\alpha<1/2$ we have  $\gg q^{\alpha l}$ values $r$ where we in turn can provide a value $k$ satisfying (\ref{thm3estimk}) which is due to
\[ 2 l+2 \leq k \leq (2+\varepsilon)l. \]
This yields
    \[ t_m(q^k) \leq  2q^{4 k + l} \leq 2q^{4 (2+\varepsilon)l + l}
       \leq q^{(9+5\varepsilon) l }. \]
Hence, letting $N = q^{(9+5\varepsilon) l }$ we note that we have
      \[
    \gg q^{\alpha l} = \left(N^{\frac{1}{(9+5\varepsilon) l}}\right)^{\alpha l} = N^{\alpha/(9+5\varepsilon)}
    \geq N^{1/19} \]
solutions to (\ref{thm3estimk}). This finishes the proof.\hfill\qed

\section{Proof of Theorem~\ref{mtheo4}}\label{sec:mtheo4}

The proof that there is only a finite number of odd $n$ such that
$s_2(n^2) = s_2(n) \leq 8$ is a strictly computational one. We
discuss how our algorithm works.

Consider $$n = \sum_{i=1}^k 2^{r_i} =2^{r_1} + 2^{r_2} + \dots + 2^{r_k}$$ with
$0 = r_1 < r_2 < r_3 < \dots < r_k$. We have
\begin{eqnarray*}
n^2 & = & \sum_{i=1}^k \sum_{j=1}^{k} 2^{r_i + r_j}=
\sum_{i=1}^k 2^{2 r_i}+\sum_{i=1}^k \sum_{j=i+1}^{k} 2^{r_i + r_j+1}.
\end{eqnarray*}
We therefore need to examine the exponents
$$\{2r_1,\; 2 r_2,\; \dots,\; 2 r_k, \;r_1+r_2+1, \;r_1 + r_3 + 1,\; \dots,\;
r_{k-1} + r_k + 1\}$$ and the possible iterations between these exponents by carry propagation.

Clearly, $2r_1$ is the strict minimum within these exponents.
Other relationships between exponents are not as clear.
For example, $r_1 + r_3 + 1$ could be less than, equal to, or greater than $2 r_2$
depending on the choices of $r_3$ and $r_2$. Each of these cases must be examined in turn.
Numerous of these inequalities have implications for the order of other exponents in the
binary expansion of $n^2$. So, once we make an assumption in our case by case analysis, this
might rule out future possibilities. For example, if we assume that $2 r_3 < 1 + r_1 + r_4$,
then we have as a consequence that $1 + r_2 + r_3 < 1 + r_1 + r_4$ (by noticing that
$r_2 < r_3$). In the case of equality we ``group'' terms. For example, if we assumed that
$2 r_3 = 1 + r_2 + r_4$, then we could, first, replace all occurrences of $r_2$ with $2 r_3 - 1 - r_4$, and
second replace $2^{2 r_3} + 2^{1 + r_2 + r_4}$ by $2^{2 r_3 + 1}$.

Our algorithm occasionally finds a solution set with fractional or negative values for $r_i$,
which is a contradiction. On the other hand, it is possible for the algorithm to find a solution,
even if all of the exponents cannot be explicitly determined. This would happen if there is an infinite family of $n$ with
$s_2(n^2) = s_2(n) = k$ with some nice structure, (as is the case for $k = 12$, see~(\ref{casek12})).
The algorithm will detect, and report this. We used the method for $k$ up to $8$. For each of these values,
there was only a finite number of $n$, and all of them are enumerated in Tables~1 and~2.

\begin{table}
\begin{tabular}{llll}
\hline
Base 10 & Base 2 & Base 10 & Base 2 \\
\hline&&&\\
\multicolumn{2}{c }{$\mathbf{s_2(n) = s_2(n^2) = 1}$} &
\multicolumn{2}{ c}{$\mathbf{s_2(n) = s_2(n^2) = 7}$} \\
1 & 1  &   127 & 1111111 \\
  &    &   319 & 100111111 \\
\multicolumn{2}{c}{$\mathbf{s_2(n) = s_2(n^2) = 2}$} &  351 & 101011111 \\
3 & 11 &    375 & 101110111 \\
  &    &    379 & 101111011 \\
\multicolumn{2}{c}{$\mathbf{s_2(n) = s_2(n^2) = 3}$} &  445 & 110111101 \\
7 & 111 &   575 & 1000111111 \\
  &     &   637 & 1001111101 \\
\multicolumn{2}{c}{$\mathbf{s_2(n) = s_2(n^2) = 4}$} &  815 & 1100101111 \\
15 & 1111 &  1087 & 10000111111 \\
   &      &  1149 & 10001111101 \\
\multicolumn{2}{c}{$\mathbf{s_2(n) = s_2(n^2) = 5}$} &  1255 & 10011100111 \\
31 &  11111 &  1815 & 11100010111 \\
79 & 1001111 & 2159 & 100001101111 \\
91 & 1011011 & 2173 & 100001111101 \\
157 & 10011101 & 2297 & 100011111001 \\
279 & 100010111 & 2921 & 101101101001 \\
   &       &   4191 & 1000001011111 \\
\multicolumn{2}{c}{$\mathbf{s_2(n) = s_2(n^2) = 6}$} &  4207 & 1000001101111 \\
63 & 111111 &  4345 & 1000011111001 \\
159 & 10011111 &  6477 & 1100101001101 \\
183 & 10110111 &  8689 & 10000111110001 \\
187 & 10111011 &  10837 & 10101001010101 \\
287 & 100011111 &  16701 & 100000100111101 \\
317 & 100111101 &  18321 & 100011110010001 \\
365 & 101101101 &  33839 & 1000010000101111 \\
573 & 1000111101 &        & \\
1071 & 10000101111 &       & \\
1145 & 10001111001 &       & \\
1449 & 10110101001 &           & \\
4253 & 1000010011101 &           & \\
4375 & 1000100010111 &           & \\
4803 & 1001011000011 &           & \\
&&&\\\hline
\end{tabular}
\caption{Odd $n$ such that $s_2(n^2) = s_2(n) \leq 7$.}
\label{tab:k=1..7}
\end{table}

\begin{table}
\begin{tabular}{llll}
\hline
Base 10 & Base 2 & Base 10 & Base 2 \\
\hline&&&\\
\multicolumn{2}{c }{$\mathbf{s_2(n) = s_2(n^2) = 8}$}   &
\multicolumn{2}{c }{$\mathbf{s_2(n) = s_2(n^2) = 8}$ \bf (cont)}
\\
255 &  11111111   &  5811 &  1011010110011 \\
639 &  1001111111 &  5865 &  1011011101001 \\
703 &  1010111111 &  5911 &  1011100010111 \\
735 &  1011011111 &  5971 &  1011101010011 \\
751 &  1011101111 &  6479 &  1100101001111 \\
759 &  1011110111 &  6557 &  1100110011101 \\
763 &  1011111011 &  8415 &  10000011011111 \\
893 &  1101111101 &  8445 &  10000011111101 \\
975 &  1111001111 &  8697 &  10000111111001 \\
1151 &  10001111111 &  10035 &  10011100110011 \\
1215 &  10010111111 &  11591 &  10110101000111 \\
1277 &  10011111101 &  11597 &  10110101001101 \\
1455 &  10110101111 &  13233 &  11001110110001 \\
1463 &  10110110111 &  13591 &  11010100010111 \\
1495 &  10111010111 &  16575 &  100000010111111 \\
1501 &  10111011101 &  16607 &  100000011011111 \\
1599 &  11000111111 &  16889 &  100000111111001 \\
1647 &  11001101111 &  17393 &  100001111110001 \\
1661 &  11001111101 &  22807 &  101100100010111 \\
2175 &  100001111111 &  23441 &  101101110010001 \\
2301 &  100011111101 &  23575 &  101110000010111 \\
2685 &  101001111101 &  25907 &  110010100110011 \\
2919 &  101101100111 &  33777 &  1000001111110001 \\
2987 &  101110101011 &  46377 &  1011010100101001 \\
3259 &  110010111011 &  46881 &  1011011100100001 \\
4223 &  1000001111111 &  51811 &  1100101001100011 \\
4349 &  1000011111101 &  66173 &  10000001001111101 \\
4601 &  1000111111001 &  67553 &  10000011111100001 \\
4911 &  1001100101111 &  69521 &  10000111110010001 \\
5069 &  1001111001101 &  133231 &  100000100001101111 \\
5231 &  1010001101111 &  227393 &  110111100001000001 \\
5799 &  1011010100111 &  266335 &  1000001000001011111 \\
&&&\\\hline
\end{tabular}
\caption{Odd $n$ such that $s_2(n^2) = s(n) = 8$.}
\label{tab:k=8}
\end{table}

\section{Proof of Theorem~\ref{mtheo5}}\label{sec:mtheo5}

For the proof of Theorem \ref{mtheo5}, we first state some
auxiliary results. Denote by $(n)_2$ the binary representation of
$n$, and $1^{(k)}$ a block of $k$ binary $1$. We begin with the
following key observation.

\begin{prop}\label{prop:infinite}
If there exists $u$ and $v$ such that $s_2(u) + s_2(v) = s_2(u^2) + s_2(u v) + s_2(v^2) = k$,
then for $i$ sufficiently large, the numbers of the form $(n)_2 = u0^i v$ satisfy
$s_2(n^2)=s_2(n)= k$.
\end{prop}

\begin{proof}
  This follows at once from Proposition~\ref{propsplit}, relation~(\ref{splitpos}).
\end{proof}

We use Proposition \ref{prop:infinite} to prove the following lemma.
\begin{lem}\label{n1n2}
Let $(u)_2 = 1^{(k_1)}01^{(n_1)}$ and $(v)_2 = 1^{(k_2)}01^{(n_2)}$. Assume
that $n_1\geq k_1+2$, $n_2\geq k_2+2$ and $n_1\ge n_2$. Then
$$ s_2(u^2) = n_1 \ \ {\rm and} \ \ s_2(v^2) = n_2,$$
and
$$
s_2(uv) =\begin{cases}
k_1+2 & {\rm if} \ n_2=k_1+1, n_1=n_2+k_2+1\\
n_2+1 & {\rm if} \ n_2>k_1+1, n_1=n_2+k_2+1\\
n_1+1 & {\rm if} \ k_1=k_2, n_1>n_2.
\end{cases}
$$
\end{lem}

\begin{proof}
 Let $(U)_2=1^{(k)}01^{(n)}$ with $n\ge k+2$. Then $U=2^n-1+2^{n+1}(2^k-1)$ and
we calculate
\begin{align*}
U^2&=2^{2n}-2^{n+1}+1+2^{n+2}(2^{n+k}-2^n-2^k+1)+2^{2n+2}(2^{2k}-2^{k+1}+1)\\
&=1+2^{n+1}+2^{2n}+2^{n+k+2}(1+2+2^2+\dots +2^{n+k-1})-2^{2n+k+2}\\
&=1+2^{n+1}+2^{n+k+2}+\dots +2^{2n-1}+2^{2n+k+2}+2^{2n+k+3}+\dots +2^{2n+2k+1}.
\end{align*}
Hence $s_2(U^2)=n$ and therefore $s_2(u^2)=n_1$ and $s_2(v^2)=n_2$.

Now, consider $s_2(uv)$. We have
\begin{align*}
uv=&1+2^{n_1}+2^{n_2}+2^{n_1+n_2}-2^{n_1+k_1+1}-2^{n_2+k_2+1}-2^{n_1+n_2+k_1+1}-\\
&2^{n_1+n_2+k_2+1}+2^{n_1+n_2+k_1+k_2+2}.
\end{align*}
We may assume that $k_1\ge k_2$. Then
\begin{align*}
W&:=2^{n_1+n_2+k_1+k_2+2}-2^{n_1+n_2+k_2+1}-2^{n_1+n_2+k_1+1}\\
&=2^{n_1+n_2+k_2+1}(1+2+\dots +2^{k_1 - k_2 - 1} + 2^{k_1 - k_2 + 1} + \dots + 2^{k_1})
\end{align*}
has $s_2(W)=k_1$. We distinguish three cases to conclude:

\begin{itemize}
\item[(1)] Let $n_1=n_2+k_2+1$ and $n_2=k_1+1$. Then $uv=1+2^{n_2}+W$ and hence
$s_2(uv)=k_1+2$.
\item[(2)]  Let $n_1=n_2+k_2+1$ and $n_2>k_1+1$. Then
$uv=1+2^{n_2}+W+2^{n_1+k_1+1}(2^{n_2-k_1-1}-1)$ and hence
$s_2(uv)=2+k_1+n_2-k_1-1=n_2+1$.
\item[(3)]  Let $k_1=k_2=k$ and $n_1>n_2$. Then
$uv=1+2^{n_2}+2^{n_1}+W+2^{n_2+k+1}(2^{n_1-k-1}-1)-2^{n_1+k+1}$ and hence
$s_2(uv)=3+k+n_1-k-2=n_1+1$.
\end{itemize}
This finishes the proof.
\end{proof}

\begin{proof}[Proof of Theorem \ref{mtheo5}]
Let $n_1, n_2, k_1, k_2$ be positive integers with $n_1\ge k_1+2$, $n_2\ge k_2+2$ and
$u, v$ be as in Lemma \ref{n1n2}. Let $(N)_2=u0^Rv$ be the binary representation of $N$ where
$R\ge n_1+n_2+k_1+k_2$. By Proposition~\ref{prop:infinite} and Lemma \ref{n1n2} we have for any
$R\ge n_1+n_2+k_1+k_2$,
\begin{align*}
  s_2(N)&=s_2(u)+s_2(v)=n_1+n_2+k_1+k_2,\\
  s_2(N^2)&=s_2(u^2)+s_2(v^2)+s_2(uv)=n_1+n_2+s_2(uv).
\end{align*}
Let $k\ge 2$. Taking $k_1 = k_2 = k$ and $n_1=n_2=2k$, we find from Lemma \ref{n1n2}
and $2k\ge k+2$ that $$s_2(N^2)=s_2(N)=6k$$ implying there are infinite families
of $n$ such that $s_2(n)=s_2(n^2)=s$ for $s$ of the form $6k$ with $k\ge 2$.

Let $k_2=2, k_1\ge 3, n_2=k_1+2$ and $n_1=n_2+k_2+1=k_1+4$. Then $s_2(uv)=n_2+1$ by Lemma \ref{n1n2}
implying $s_2(N^2)=s_2(N)=3(k_1+2)+1$. Hence there are infinite families of $n$ such that
$s_2(n)=s_2(n^2)=s$ for $s$ of the form $3k+1$ with $k\ge 5$.

Let $k_1\ge k_2\ge 3$ and $n_2=k_1+k_2-1, n_1=n_2+k_2+1$. Then $s_2(uv)=n_2+1=k_1+k_2$ from
Lemma \ref{n1n2} implying
$s_2(N^2)=s_2(N)=3k_1+4k_2-1$. Let $k_2=3$. Then $s_2(N^2)=s_2(N)=3(k_1+3)+2$ for $k_1\ge 3$ giving
infinite families of $n$ such that $s_2(n)=s_2(n^2)=s$ for $s$ of the form $3k+2$ with $k\ge 6$.

Let $k_2=4$. Then $s_2(N^2)=s_2(N)=3(k_1+5)$ for $k_1\ge 4$ giving
infinite families of $n$ such that $s_2(n)=s_2(n^2)=s$ for $s$ of the form $3k$ with $k\ge 27$.

Summing up, we have infinite families of $n$ with $s(n^2)=s(n)=s$ for all $s\ge 22$, respectively,
$s\in \{12, 16, 18, 19, 20\}$. For $s\in \{13, 17, 21\}$, we take $(N)_2=u0^Rv$ with
\begin{align*}
s=13: \ &u=10111, \  v=10110111111\\
s=17: \ &u=111011111, \  v=10110111111\\
s=21: \ &u=11110111111, \ v=111101111111.
\end{align*}

This completes the proof of Theorem~\ref{mtheo5}.
\end{proof}

\section{Evidence that $s_2(n^2) = s_2(n) \leq 10$ is finite}
\label{sec:comp2}

All examples of infinite families with $s_2(n^2) = s_2(n) = k$
have the form  given from Lemma~\ref{prop:infinite}. We show that
there do not exists $u$ and $v$ satisfying
Proposition~\ref{prop:infinite}, with $k\in\{9, 10\}$. We
illustrate this method for $k = 8$, as it contains all of the key
ideas without being overly cumbersome. The case of $k = 8$ is
actually proved to be finite by the techniques of Section
\ref{sec:mtheo4}, but this does not detract from this example. The
other two cases are similar.

Assume the contrary, that there exists $u$ and $v$ such that
\[ s_2(u) + s_2(v) = s_2(u^2) + s_2(v^2) + s_2(u v) = 8 \]
We easily see that $s_2(v), s_2(u) \geq 2$. Furthermore, as $s_2(u), s_2(v) \geq 2$,
we see that $s_2(u^2), s_2(v^2) \geq 2$. Also, we have that $s_2(u v) \geq 2$.
Therefore, we have $2 \leq s_2(u^2), s_2(v^2) \leq k - 4$. Lastly, we see that one
of $u$ or $v$ must be ``deficient'', that $s_2(u^2)<s_2(u)$ or $s_2(v^2)<s_2(v)$.

Assume without loss of generality that $s_2(u^2) < s_2(u)$. Given the restrictions,
we have that $2 \leq s_2(u) \leq 6$. Using the same algorithm as in Section~\ref{sec:mtheo4},
we can find all $u$ such that $2 \leq s_2(u) \leq 6$ and $s_2(u^2) < s_2(u)$, $s_2(u^2) \leq 4$.
These are the first three entries of Table~3.

\begin{table}
\begin{tabular}{llll}
\hline
Base 10 & Base 2 &  &  \\
\hline&&&\\
\multicolumn{2}{c }{$u$} & $s_2(u)$ & $s_2(u^2)$ \\
23& 10111& 4& 3 \\ 
47& 101111& 5& 4 \\
111& 1101111& 6& 4 \\
95& 1011111& 6& 5 \\ 
5793& 1011010100001& 6& 5 \\
223& 11011111& 7& 5 \\ 
727& 1011010111& 7& 5 \\
191& 10111111& 7& 6 \\ 
367& 101101111& 7& 6 \\ 
415& 110011111& 7& 6 \\ 
1451& 10110101011& 7& 6 \\ 
46341& 1011010100000101& 7& 6 \\ 
479& 111011111& 8& 5 \\ 
447& 110111111& 8& 6 \\ 
887& 1101110111& 8& 6 \\
&&&\\\hline
\end{tabular}
\caption{$s_2(u) \leq 8$, $s_2(u^2) < s_2(u)$ and $s_2(u^2) \leq 6$.}
\end{table}

Therefore, it suffices to show that there do not exists $v$ for $u = 23, 47$ or $111$
with $s_2(u) + s_2(v) = s_2(u^2) + s_2(v^2) + s_2(uv) = 8$.

\begin{enumerate}
\item[(1)] Let $u = 23 = 10111$. Given that $s_2(uv) \geq 2$ we have that $s_2(v) = 4$ and $s_2(v^2) \leq 3$.
The only possible solution by Table~3 is $v = 23 = 10111$, but $s_2(uv) = 3$, a contradiction.

\item[(2)] Let $u = 47 = 101111$. Given that $s_2(uv) \geq 2$ we have that $s_2(v) = 3$ and $s_2(v^2) \leq 2$.
There are no solutions by Table~3 for this, a contradiction.

\item[(3)] Let $u = 111 = 11101111$. Given that $s_2(uv) \geq 2$ we have that $s_2(v) = 2$ and $s_2(v^2) \leq 2$.
There is one possible solution to this by Table~3, namely $v = 3 = 11$. But then $s_2(u v) = 5$, a contradiction.
\end{enumerate}

A similar, but more elaborate analysis can be done for $k = 9$ and $k = 10$
using the additional information in Table~3. Here we look at $2 \leq s_2(u) \leq 7$,
$s_2(u^2) < s_2(u)$ and $s_2(u^2) \leq 5$.

\section{Proof of Theorem~\ref{genqsquares}}
\label{sec:genqsquares}

The proof uses the strategy adopted for the case $q=2$ (see
Section~\ref{sec:mtheo4}). However, in order to handle more
possible digits in the case of $q\geq 3$, the analysis is
much more delicate. In the proof we will make
frequent use of the fact~(\ref{wellknown})
and of the splitting formulae of Proposition~\ref{propsplit},
which will apply if we have noninterfering terms at our disposal.

To begin with, the condition~(\ref{knec}) is necessary,
since~(\ref{Dioeq}) implies
$$s_q(n^2)-s_q(n)\equiv n^2-n \equiv k^2 -k \equiv 0\quad \bmod (q-1).$$
For the construction of an infinite family, we first prove a crucial lemma.
\begin{lem}
  Let $$u=((q-1)^k \; 0 \; (q-1)^n e)_q$$ with $k\geq 2$, $n\geq k+2$ and $0\leq e\leq q-2$.
  Then
  $$s_q(u)=(q-1)(n+k)+e$$
  and
  $$s_q(u^2)=(q-1)(n+1)+f(q,e)$$
  where
  \begin{equation}\label{fdefini}
    f(q,e)=s_q((q-e)^2)+s_q(2(q-1)(q-e))-s_q(2(q-e)-1).
  \end{equation}
\end{lem}
\begin{proof}
  Since $u=e+(q^n-1)q+(q^k-1) q^{n+2}$, we get
  \begin{align}
    u^2 &= (q-e)^2+ 2 (q-1)(q-e) q^{n+1}-2(q-e)q^{n+k+2}\nonumber\\
   &\quad +(q-1)^2 q^{2n+2}-2(q-1)q^{2n+k+3}+q^{2n+2k+4}.\label{u2exp}
  \end{align}
   By assumption that $n\geq k+2$ and $n,k\geq 2$, the terms in~(\ref{u2exp}) are noninterfering.
   We therefore get
  \begin{align*}
    s_q(u^2) &= s_q((q-e)^2)+s_q(2(q-1)(q-e))-s_q(2(q-e)-1)+(n-k)(q-1)\\
    &\qquad +s_q((q-1)^2-1)-s_q(2(q-1)-1)+(k+1)(q-1).\\
    & = (n+1)(q-1)+s_q(q^2-2q)-s_q(2q-3)+f(q,e).
  \end{align*}
  The claimed value of $s_q(u^2)$ now follows by observing that $s_q(q^2-2q)=s_q(q-2)=q-2$ and
$s_q(2q-3)=s_q(q+q-3)=1+q-3=q-2$.
\end{proof}

Now  consider
\begin{align*}
  u&=((q-1)^{k_1} \; 0 \; (q-1)^{n_1})_q,\\
  v&=((q-1)^{k_2} \; 0 \; (q-1)^{n_2} e)_q
\end{align*}
where we suppose $k_1, n_1, k_2, n_2 \geq 2$ and $n_1\geq k_1+2$, $n_2 \geq k_2+2$. Since $q\nmid n$ we further suppose that $e\neq 0$. We want to
construct an infinite family of solutions to~(\ref{Dioeq}) of the form $n= (u 0^{(i)} v)$, where $i$ is a sufficiently large integer, such that terms will
be noninterfering. Our task is to find an admissible set of parameters
$k_1, n_1, k_2, n_2$ such that for sufficiently large $n_1+n_2+k_1+k_2$ we have
\begin{align}
  s_q(u)+s_q(v)&=s_q(u^2)+s_q(2uv)+s_q(v^2)\nonumber\\
  &= e+(q-1)(n_1+n_2+k_1+k_2).\label{dioiden}
\end{align}
First it is a straightforward calculation to show that $2uv=w_1+w_2$ with
\begin{equation}\label{w1def}
  w_1 = 2q^{n_1+n_2+k_1+k_2+3} - 2(q-1)q^{n_1+n_2+k_1+2}- 2(q-1)q^{n_1+n_2+k_2+2}
\end{equation}
and
\begin{align}\label{w2def}
  w_2 &= 2(q-1)^2 q^{n_1+n_2+1}-2 (q-e)q^{n_1+k_1+1}-2q^{n_2+k_2+2}\nonumber\\
  &\quad +2(q-1)(q-e)q^{n_1}+ 2 (q-1)q^{n_2+1}+2(q-e).
\end{align}
Note that $w_1$ and $w_2$ are noninterfering because of $k_2\geq 2$. Now, set
\begin{equation}\label{crucialchoix}
  k_1=n_2\geq k_2+2,\qquad n_1=2k_2-\alpha,
\end{equation}
where we will later suitably choose $\alpha=\alpha(q,e)$ only depending on $q$ and $e$.
Then terms in~(\ref{w1def}) are again noninterfering and we get
\begin{align*}
  s_q(w_1)&=s_q(2q^{k_1+1}-2(q-1)q^{k_1-k_2}-2(q-1))\\
  & = s_q(2q^{k_2+1}-2q+1)+(q-1)(k_1-k_2)-s_q(2q-3)\\
  & = 1+k_2(q-1)+(q-1)(k_1-k_2)-(q-2)\\
  & =(k_1-1) (q-1)+2.
\end{align*}
Next, by~(\ref{crucialchoix}), we find that
\begin{align}\label{w2defdot}
  w_2 &= 2 q^{k_1+2k_2-\alpha+1} ((q-1)^2-(q-e))-2 q^{k_1+k_2+2}\nonumber\\
  &\quad +2(q-1)(q-e)q^{2 k_2-\alpha}+ 2 (q-1)q^{k_1+1}+2(q-e).
\end{align}
In order to have terms noninterfering in~(\ref{w2defdot}), we impose the following inequalities on the parameters,
\begin{align}
  2&\leq k_1+1,\label{ineq1}\\
  2&\leq (2k_2-\alpha)-(k_1+1),\label{ineq2}\\
  3&\leq (k_1+k_2+2)-(2k_2-\alpha)=k_1-k_2+2+\alpha,\label{ineq3}\\
  1&\leq (k_1+2k_2-\alpha+1)-(k_1+k_2+2)=k_2-\alpha-1.\label{ineq4}
\end{align}
Then we get
$$s(w_2)=(k_2-\alpha-1)(q-1)+g(q,e)$$
where
\begin{align}
  g(q,e)&=s_q(2(q-e))+s_q(2(q-1))+s_q(2(q-1)(q-e))\nonumber\\
  &\qquad +s_q(2(q-1)^2-(q-e)-1)-1.\label{gdefini}
\end{align}
Summing up, we have
\begin{align*}
  &s_q(u^2)+s_q(2uv)+s_q(v^2)\\
&\quad=(q-1)(n_1+1)+f(q,e)+(q-1)(n_2+1)+(k_1-1)(q-1)\\
  &\quad\qquad +2+(k_2-\alpha-1)(q-1)+g(q,e)\\
  &\quad=(q-1)(2 k_1+3 k_2-2\alpha)+f(q,e)+g(q,e)+2.
\end{align*}
Combining with~(\ref{dioiden}) and~(\ref{crucialchoix}) we therefore have
\begin{equation}\label{egalite}
  (q-1)(2 k_1+3 k_2-2\alpha)+f(q,e)+g(q,e)+2=(q-1)(2 k_1+3 k_2-\alpha)+e
\end{equation}
and
$$\alpha(q-1)=f(q,e)+g(q,e)-e+2.$$
Rule~(\ref{wellknown}) applied to~(\ref{fdefini})
and~(\ref{gdefini}) shows that the right hand side is indeed
divisible by $q-1$ since $e^2-e \equiv 0$ mod $(q-1)$ by
assumption. Furthermore, we have by a crude estimation (using
also~(\ref{wellknown})) that
\begin{equation}\label{alphaestim}
  0\leq \alpha\leq 15.
\end{equation}
Suppose $k_2\geq 17$. Then~(\ref{ineq1}) and~(\ref{ineq4}) are satisfied. Rewriting~(\ref{ineq2}) and~(\ref{ineq3}) gives
\begin{equation}\label{k1ineq}
  1+k_2-\alpha\leq k_1\leq 2k_2-\alpha-1.
\end{equation}
Note that $k_1\geq k_2+2$ is more restrictive than the first inequality in~(\ref{k1ineq}). On the other hand, since $k_2\geq 2$, the interval given for $k_1$ in~(\ref{k1ineq}) has at least
$(2\cdot 17-\alpha-1)-(1+17-\alpha)+1=16$ terms. Therefore, $2k_1+3k_2$ hits all integers $\geq 2 (1+(k_2+1)-\alpha)+3(k_2+1)$ for $k_2\geq 17$.
Thus, we find from~(\ref{egalite}) that
all values
\begin{align*}
  (q-1)(2 k_1+3 k_2-\alpha)+e &\geq (q-1)(2 \cdot (19-0)+3\cdot 18)+(q-1)\\
  & = 94 (q-1)
\end{align*}
can be achieved. This completes the proof of Theorem~\ref{genqsquares}.

\end{document}